\documentclass[10pt]{amsart}
\usepackage{rotating,latexsym,amsfonts,amsmath,enumerate,enumerate,
amsthm}
\newtheorem{theorem}{Theorem}[section]
\newtheorem{lemma}[theorem]{Lemma}

\newtheorem{corollary}[theorem]{Corollary}

\theoremstyle{definition}

\newtheorem{algorithm}[theorem]{Algorithm}

\theoremstyle{remark}

\numberwithin{equation}{section}

\begin{document}
\title[On the complexity of detecting positive eigenvectors]{On the complexity of detecting positive eigenvectors of nonlinear cone maps}

\author{Bas Lemmens}
\address{School of Mathematics, Statistics \& Actuarial Science, Sibson Building, 
University of Kent, Canterbury, Kent CT2 7FS, UK}
\curraddr{}
\email{B.Lemmens@kent.ac.uk}
\thanks{}

\author{Lewis White}
\address{School of Mathematics, Statistics \& Actuarial Science, Sibson Building, 
University of Kent, Canterbury, Kent CT2 7FS, UK}
\curraddr{}
\email{lcw32@kent.ac.uk}
\thanks{The second author was supported by a London Mathematical Society ``Undergraduate Research Bursary'' and the 
School of Mathematics, Statistics and Actuarial Science at the University of Kent.}



\dedicatory{}

\begin{abstract}
In recent work with Lins and Nussbaum the first author gave an algorithm that can detect the existence of a positive eigenvector for order-preserving homogeneous maps on  the standard positive cone.   The main goal of this paper is to determine the minimum  number of iterations this algorithm requires. It is known that  this number is equal to the illumination number of the unit ball, $B_{\mathrm{v}}$, of the variation norm, $\|x\|_{\mathrm{v}} :=\max_i x_i -\min_i x_i$ on $V_0:=\{x\in\mathbb{R}^n\colon x_n=0\}$. In this paper we   show that the illumination number of $B_{\mathrm{v}}$ is equal to ${n\choose\lceil \frac{n}{2}\rceil}$, and hence provide a sharp lower bound for the running time of the algorithm. 
\end{abstract}

\maketitle
\section{Introduction}
The classical Perron-Frobenius theory concerns the spectral properties of square nonnegative matrices. In recent decades this theory has been extended to a variety of nonlinear maps  that preserve a partial ordering induced by a cone (see \cite{LNBook} and the references therein for an up-to-date account). 

Of particular interest  are  order-preserving homogeneous maps $f\colon \mathbb{R}^n_{\geq 0}\to  \mathbb{R}^n_{\geq 0}$, where 
\[ \mathbb{R}^n_{\geq 0}:=\{x\in\mathbb{R}^n\colon x_i\geq 0\mbox{ for all } i=1,\ldots,n\}\] is the {\em standard positive cone}. Recall that $f\colon  \mathbb{R}^n_{\geq 0}\to  \mathbb{R}^n_{\geq 0}$ is {\em order-preserving} if $f(x)\leq f(y)$ whenever  $x\leq y$ and  $x,y\in  \mathbb{R}^n_{\geq 0}$. Here $w\leq z$ if $z-w\in  \mathbb{R}^n_{\geq 0}$. Furthermore, $f$ is said to be {\em homogeneous} if $f(\lambda x) =\lambda f(x)$ for all  $\lambda \geq 0$ and $x\in  \mathbb{R}^n_{\geq 0}$. Such maps arise in mathematical biology \cite{NMem2,Sch} and in optimal control and game theory \cite{BK,RS}. 

It is known \cite[Corollary 5.4.2]{LNBook} that if $f\colon  \mathbb{R}^n_{\geq 0}\to  \mathbb{R}^n_{\geq 0}$ is a continuous, order-preserving, homogeneous map, then there exists $v\in  \mathbb{R}^n_{\geq 0}$ such that 
\[
f(v) = r(f) v,
\]
where 
\[
r(f) :=\lim_{k\to\infty} \|f^k\|^{1/k}_{ \mathbb{R}^n_{\geq 0}} 
\]
is the {\em cone spectral radius} of $f$ and \[
\|g\|_{\mathbb{R}^n_{\geq 0}}:=\sup \{\|g(x)\|\colon x\in \mathbb{R}^n_{\geq 0}\mbox{ and }  \|x\|\leq 1\}.\]
Thus, as in the case of nonnegative matrices, continuous order-preserving  homogeneous maps on $\mathbb{R}^n_{\geq 0}$ have an eigenvector in the cone corresponding to the spectral radius. 

In many applications it is important to know if the map has a {\em positive} eigenvector, i.e., an eigenvector that lies in the 
interior, $\mathbb{R}^n_{>0} :=\{x\in\mathbb{R}^n_{\geq 0}\colon x_i > 0\mbox{ for }i=1,\ldots,n\}$, of $\mathbb{R}^n_{\geq 0}$. This appears  to be a much more subtle problem.  There exists a variety of sufficient conditions in the literature, see \cite{Ca}, \cite{GG}, \cite[Chapter 6]{LNBook}, and \cite{NMem1}. Recently,  Lemmens, Lins and Nussbaum  \cite[Section 5]{LLN}  gave an algorithm that can confirm the existence of a positive eigenvector for continuous, order-preserving, homogeneous maps $f\colon \mathbb{R}^n_{\geq 0}\to \mathbb{R}^n_{\geq 0}$. The main goal of this paper is to determine the minimum number of iterations this algorithm needs to perform.   

\section{Preliminaries} 
 Given a set $S$ in a finite dimensional vector space $V$ we write $S^\circ$ to denote the interior of $S$, and we write $\partial S$ to denote the boundary of $S$ with respect to the norm topology on $V$.  

It is known that if $f\colon \mathbb{R}^n_{\geq 0}\to \mathbb{R}^n_{\geq 0}$ is an order-preserving homogeneous map and there exists $z\in \mathbb{R}^n_{> 0}$ such that $f(z) \in\partial \mathbb{R}^n_{\geq 0}$, then $f(\mathbb{R}^n_{> 0})\subset \partial \mathbb{R}^n_{\geq 0}$, see \cite[Lemma 1.2.2]{LNBook}. Thus to analyse the existence of a positive eigenvector one may as well consider order-preserving homogeneous maps $f\colon \mathbb{R}^n_{>0}\to \mathbb{R}^n_{> 0}$. Moreover, on $\mathbb{R}^n_{>0}$ we have {\em Hilbert's metric}, $d_H$, which is given by 
\[
d_H(x,y) := \log \left(\max_i \frac{x_i}{y_i}\right) - \log \left(\min_i \frac{x_i}{y_i}\right)\mbox{\quad for }x,y\in \mathbb{R}^n_{>0}.
\]
Note that $d_H$ is not a genuine metric, as $d_H(\lambda x, \mu x) = 0$ for all $x\in  \mathbb{R}^n_{>0}$ and $\lambda,\mu >0$. In fact,  $d_H(x,y) =0$ if and only if $x=\lambda y$ for some $\lambda >0$. However, $d_H$ is a metric on the set of rays in $\mathbb{R}^n_{>0}$. 

If $f\colon \mathbb{R}^n_{>0}\to \mathbb{R}^n_{>0}$ is order-preserving and homogeneous, then $f$ is nonexpansive under $d_H$, i.e., 
\[
d_H(f(x),f(y))\leq d_H(x,y)\mbox{\quad for all }x,y\in \mathbb{R}^n_{>0},  
\]
see for example \cite[Proposition 2.1.1]{LNBook}. 
In particular, order-preserving homogeneous maps $f\colon \mathbb{R}^n_{>0}\to \mathbb{R}^n_{>0}$ are continuous on $\mathbb{R}^n_{>0}$.  
Moreover, if $x$ and $y$ are eigenvectors of $f\colon \mathbb{R}^n_{>0}\to \mathbb{R}^n_{>0}$ with $f(x) = \lambda x$ and $f(y)=\mu y$, then $\lambda =\mu$, see \cite[Corollary 5.2.2]{LNBook}. 

In \cite[Theorem 5.1]{LLN} the following necessary and sufficient conditions were obtained for an order-preserving homogeneous map $f\colon  \mathbb{R}^n_{>0}\to \mathbb{R}^n_{>0}$ to have a nonempty set of eigenvectors, $\mathrm{E}(f) :=\{x\in  \mathbb{R}^n_{> 0}\colon x\mbox{ eigenvector of } f\}$, which is bounded under Hilbert's metric. 
\begin{theorem}\label{thm:npfthm} If $f\colon \mathbb{R}^n_{> 0}\to \mathbb{R}^n_{> 0}$ is an order-preserving homogeneous map, then $\mathrm{E}(f)$ is nonempty and bounded under $d_H$ if and only if for each nonempty proper subset $J$ of $\{1,\ldots,n\}$ there exists $x^J\in \mathbb{R}^n_{> 0}$ such that 
\begin{equation}\label{eq:1.1}
\max_{j\in J}\,\frac{f(x^J)_j}{x^J_j}< \min_{j\in J^c}\,\frac{f(x^J)_j}{x^J_j}.
\end{equation}
\end{theorem} 
Note that the assertion is trivial in case $n=1$, as each order-preserving homogeneous map $f\colon \mathbb{R}_{> 0}\to \mathbb{R}_{> 0}$ has a nonempty bounded set of eigenvectors.
In case $n\geq 2$ Theorem \ref{thm:npfthm}  yields the following  simple  algorithm for detecting positive eigenvectors:  
\begin{algorithm} \label{alg} Let $f\colon \mathbb{R}^n_{>0} \to \mathbb{R}^n_{>0}$ be an order-preserving homogeneous  map.  
Repeat the following steps until every nonempty proper subset $J$  of $\{1,\ldots, n\}$ has been recorded. 
\begin{description}
\item[Step 1] Randomly select $x$, with $x_1=1$ and $0<x_j<1$ for all $j\in\{2,\ldots,n\}$, and compute $f(x)_j/x_j$ for all $j \in \{1,\ldots, n\}$.  
\item[Step 2] Record all nonempty proper subsets $J \subset \{1,\ldots,n\}$ such that inequality  (\ref{eq:1.1}) holds.   
\end{description}
\end{algorithm}
So,  if this  algorithm halts, then $f$ has an eigenvector in $\mathbb{R}^n_{>0}$ and $\mathrm{E(}f)$ is bounded under Hilbert's metric. If $\mathrm{E}(f)$ is empty or unbounded under $d_H$, then the algorithm does not halt. This can happen even if the map is linear. Consider, for example the linear map $x\mapsto Ax$ on $\mathbb{R}^2_{>0}$, where 
\[
A = \left [\begin{array}{cc} 1& 1 \\ 0 & 1 \end{array}\right], 
\]
which has no eigenvector in $\mathbb{R}^2_{>0}$.  At present no algorithm is known that can decide if an order-preserving homogeneous map on  $\mathbb{R}^n_{>0}$ has an empty or an unbounded set of eigenvectors. It is also unknown if there is an efficient way to generate the vectors $x$ in Step 1. 

Note that a randomly chosen $x$ in Step 1 can eliminate multiple subsets $J$ in Step 2. So, it is natural to ask for the least number of vectors required to fulfill  the  $2^n-2$ inequalities in (\ref{eq:1.1}). This number corresponds to the minimum number of times the algorithm has to perform Steps 1 and 2. In this  paper we  show that one needs at least 
\[
n\choose \lceil n/2\rceil
\]
vectors and  this lower bound is sharp. Here $\lceil a\rceil$ is the smallest integer $n\geq a$. Likewise we  write $\lfloor a\rfloor $ to denote the largest integer $n\leq a$. 

\section{Connection with the illumination number}
Recall that given a compact convex set $C$ with nonempty interior in $V$, a vector $v\in V$ {\em illuminates} $z\in\partial  C$ if $z+\lambda v\in C^\circ$ for all $\lambda>0$ sufficiently small. A set $S$ is said to {\em illuminate} $C$ if for each $z\in\partial C$ there exists $v\in S$ such that $v$ illuminates $z$. The minimal size of illuminating set for $C$ is called the {\em illumination number} of $C$ and is denoted $i(C)$. There is a long-standing open conjecture which asserts that $i(C)\leq 2^n$ for every compact convex body in an $n$-dimensional vector space, see \cite[Chapter VI]{BMS} for further details. It is easy to show, see for example \cite[Lemma 4.1]{LLN}, that if $S$ illuminates every extreme point of $C$, then $S$ illuminates $C$. 

To proceed we need to discuss the connection between illumination numbers and  Theorem \ref{thm:npfthm}. 
Firstly, we note that if we let $\Sigma_0:=\{x\in \mathbb{R}^n_{>0}\colon x_n=1\}$, then $(\Sigma_0,d_H)$ is a metric space. Given an order-preserving  homogeneous map $f\colon \mathbb{R}^n_{>0} \to \mathbb{R}^n_{>0}$ we can consider the {\em normalised} map $g_f\colon \Sigma_0\to \Sigma_0$ given by 
\[
g_f(x) := \frac{f(x)}{f(x)_n}\mbox{ \quad for }x\in \Sigma_0.
\]
The map $g_f$ is nonexpansive  under $d_H$ on $\Sigma_0$. Moreover, $x\in \Sigma_0$ is a fixed point of $g_f$ if and only if $x$ is an eigenvector of $f$. Thus, if we let $\mathrm{Fix}(g_f) :=\{x\in\Sigma_0\colon g_f(x) =x\}$, then $\mathrm{Fix}(g_f)$ is nonempty and bounded in $(\Sigma_0,d_H)$ if and only if $\mathrm{E}(f)$ is nonempty and bounded in $(\mathbb{R}^n_{>0},d_H)$.  

It not hard to verify that the map $\mathrm{Log}\colon \Sigma_0\to V_0$ given by 
\[
\mathrm{Log}(x) := (\log x_1,\ldots,\log x_n)\mbox{\quad for }x=(x_1,\ldots,x_n)\in\Sigma_0
\]
is an isometry from $(\Sigma_0,d_H)$ onto $(V_0,\|\cdot\|_{\mathrm{v}})$, where $V_0 :=\{x\in\mathbb{R}^n\colon x_n =0\}$ and 
\[
\|x\|_{\mathrm{v}} := \max_i x_i -\min_i x_i
\]
is the {\em variation norm}. 

It follows that the map $h\colon V_0\to V_0$ satisfying $h\circ \mathrm{Log} = \mathrm{Log}\circ g_f$ is nonexpansive under the variation norm, and $\mathrm{Fix}(h)$ is nonempty and bounded in $(V_0,\|\cdot\|_{\mathrm{v}})$ if and only if $\mathrm{Fix}(g_f)$ is nonempty and bounded in $(\Sigma_0,d_H)$. 

In \cite[Theorem 3.4]{LLN}  the following result concerning fixed point sets of nonexpansive maps on finite dimensional normed spaces was proved.
\begin{theorem}\label{thm:fp} If $h\colon V\to V$ is a  nonexpansive map  on a finite dimensional normed space $V$, then $\mathrm{Fix}(h)$  is  nonempty and  bounded  if and only if there exist $w^1,\ldots,w^m\in V$ such that $\{f(w^i)-w^i\colon i=1,\ldots,m\}$ illuminates the unit ball of $V$. 
\end{theorem}

For $n\geq 2$, the unit ball $B_\mathrm{v}$ of $(V_0,\|\cdot\|_{\mathrm{v}})$ has $2^n-2$ extreme points, which are given by 
  \begin{equation} \label{eq:varExtr}
\mathrm{ext}(B_{\mathrm{v}}) :=\{v^I_+\colon \emptyset \neq I\subseteq \{1,\ldots,n-1\}\}\cup 
\{v^I_-\colon \emptyset \neq I\subseteq \{1,\ldots,n-1\}\}, 
\end{equation}
where $(v^I_+)_i =1$ if $i\in I$ and $0$ otherwise, and $(v^I_-)_i =-1$ if $i\in I$ and $0$ otherwise. See \cite[\S 2]{Nu2} for details.

In \cite{LLN} the equivalence in Theorem \ref{thm:npfthm} was obtained by using Theorem \ref{thm:fp} and showing that there exists $x^1,\ldots,x^m\in \mathbb{R}^n_{>0}$ that fulfill the $2^n-2$ inequalities in (\ref{eq:1.1}) if and only if there exist $y^1,\ldots,y^m\in V_0$ that illuminate the   $2^n-2$ extreme points of the unit ball $B_\mathrm{v}$.  
Thus, $i(B_{\mathrm{v}})$ provides a sharp lower bound for the number of times one needs to repeat Steps 1 and 2 in  Algorithm \ref{alg}.  In the next section we  show the following result concerning $i(B_{\mathrm{v}})$. 
\begin{theorem}\label{thm:main}
If $B_\mathrm{v}$ is the unit ball of $(V_0,\|\cdot\|_{\mathrm{v}})$ and $n\geq 2$, then 
\[
i(B_\mathrm{v}) = {n\choose\lceil n/2\rceil}.
\] 
 \end{theorem}

\section{Proof of Theorem \ref{thm:main}}
Note that the map $(x_1,\ldots,x_n)\in V_0\mapsto (x_1,\ldots,x_{n-1})\in\mathbb{R}^{n-1}$ is an isometry from $(V_0,\|\cdot\|_{\mathrm{v}})$ onto 
$(\mathbb{R}^{n-1},\|\cdot\|_H)$, where 
\[
\|x\|_H :=\left(\max_i x_i\right) \vee 0 - \left(\min_i x_i\right) \wedge 0.
\]
Here $a\wedge b:= \min(a,b)$ and $a\vee b:=\max(a,b)$.  Note also that if $B_H$ is the unit ball in $(\mathbb{R}^{n-1},\|\cdot\|_H)$, then 
\[
\mathrm{ext}(B_H) = \left(\{0,1\}^{n-1}\cup \{0,-1\}^{n-1}\right)\setminus\{(0,\ldots,0)\}
\]
and \[i(B_H) = i(B_{\mathrm{v}}).\]  For notational simplicity we  work with $B_H$ instead of $B_{\mathrm{v}}$. 

The following two subsets, 
\[
E_+:= \{0,1\}^{n-1}\setminus\{(0,\ldots,0)\}\mbox{\quad and\quad} E_-:=\{0,-1\}^{n-1}\setminus\{(0,\ldots,0)\},
\]
of $\mathrm{ext}(B_H)$ play a key role in the argument.  On $\mathrm{ext}(B_H)$ we have the usual partial ordering $x\leq y $ if $y-x\in\mathbb{R}^{n-1}_{\geq 0}$, which gives rise to two finite partially ordered sets $(E_+,\leq)$ and $(E_-,\leq)$. 

Recall that subset $\mathcal{A}$ of a partially ordered set $(P,\preceq)$ is called an {\em antichain} if $x,y\in \mathcal{A}$ and $x\preceq y$ implies $x=y$. A {\em chain} $\mathcal{C}$ in $(P,\preceq)$ is a totally ordered subset, if for each $x,y\in \mathcal{C}$ we have that either $x\preceq y$ or $y\preceq x$. The {\em length} of  a chain $\mathcal{C}$ is the number of distinct elements in $\mathcal{C}$.

\begin{lemma}\label{lem:antichain}
Let $\mathcal{A}$ be an antichain in  $(E_+,\leq)$ or in  $(E_-,\leq)$. If $x\neq y$ in $\mathcal{A}$ are illuminated by $v$ and $w$, respectively, then $v\neq w$. 
\end{lemma}
\begin{proof}
Suppose that $\mathcal{A}$ is antichain in  $(E_+,\leq)$ and $x\neq y$ are in $\mathcal{A}$. Then there exist $i\neq j$ such that $0=x_i< y_i =1$ and $0=y_j<x_j =1$. 
Now suppose by way of contradiction that  $z$ illuminates $x$ and $y$. So, $\|x+\lambda z\|_H<1$ and $\|y+\lambda z\|_H<1$ for all $\lambda>0$ sufficiently small. Suppose first that $z_i\leq z_j$. Then for $\lambda>0$ small, 
\[
1+\lambda z_j =x_j +\lambda z_j \leq \|x+\lambda z\|_H<1, 
\]
and hence $z_j<0$. So, $z_i\leq z_j<0$. But then 
\[
1+ \lambda(z_j-z_i) = x_j+\lambda z_j - \lambda z_i \leq \|x+\lambda z\|_H<1,
\]
which is impossible.  On the other hand, if $z_j\leq z_i$, then $1+\lambda z_i \leq \|y+\lambda z\|_H<1$, so that $z_j\leq z_i<0$. But then 
\[
1+ \lambda(z_i-z_j) = y_i+\lambda z_i- \lambda z_j \leq \|y+\lambda z\|_H<1,
\]
which again is impossible.  Thus, $z$ cannot illuminate both $x$ and $y$.  

The argument for the case where $\mathcal{A}$ is antichain in  $(E_-,\leq)$ is similar.
\end{proof}

\begin{lemma} \label{lem:+-1} If $x,y\in \mathrm{ext}(B_H)$ are such that $x_i=1$ and $y_i=-1$ for some $i$, then one needs two distinct vectors to illuminate $x$ and $y$.
\end{lemma}
\begin{proof}
Suppose $w$ illuminates $x$ and $y$. Then $1+\lambda w_i  = x_i +\lambda w_i \leq \|x+\lambda w\|_H<1$ for all $\lambda>0$ sufficiently small, and hence $w_i<0$. But also 
 $1 - \lambda w_i = -( y_i +\lambda w_i) \leq \|y+\lambda w\|_H< 1$ for all $\lambda>0$ sufficiently small. This implies that $w_i>0$, which is impossible. Thus, one needs at least two vectors to  illuminate $x$ and $y$. 
\end{proof}

\begin{corollary}\label{lowerbnd}
If $B_H$ is the unit ball of $(\mathbb{R}^{n-1},\|\cdot\|_H)$ and $n\geq 2$, then 
\[
i(B_H) \geq {n\choose\lceil n/2\rceil}.
\] 
\end{corollary}
\begin{proof}
For $1\leq k,m\leq n-1$ define the antichians $\mathcal{A}_+(k):= \{ x\in E_+\colon \sum_i x_i = k\}$ and $\mathcal{A}_-(m):= \{ x\in E_-\colon \sum_i x_i = -m\}$.
If $n>1$ is odd, then we can take $k := (n-1)/2$ and $m:=(n+1)/2$ and conclude from  Lemmas \ref{lem:antichain} and \ref{lem:+-1}  that we need at least 
\[
{n-1\choose \frac{n-1}{2}} + {n-1\choose \frac{n+1}{2}} = {n\choose \lceil \frac{n}{2}\rceil}
\] 
distinct vectors to illuminate the extreme points in $\mathcal{A}_+(k)\cup \mathcal{A}_-(m)$, as for each $x\in\mathcal{A}_+(k)$ and $y\in\mathcal{A}_-(m)$ there exists an $i$ such that $x_i =1$ and $y_i=-1$. 

Likewise if $n>1$ is even, we can take $k = m= \lceil \frac{n-1}{2}\rceil$, and deduce from Lemmas \ref{lem:antichain} and \ref{lem:+-1}  that we need at least 
\[
{n-1\choose \lceil \frac{n-1}{2}\rceil} + {n-1\choose \lceil\frac{n-1}{2}\rceil} = {n-1\choose \lfloor \frac{n-1}{2}\rfloor} + {n-1\choose \lceil\frac{n-1}{2}\rceil}  ={n\choose  \frac{n}{2}}
\] 
distinct vectors to illuminate the extreme points in $\mathcal{A}_+(k)\cup \mathcal{A}_-(m)$. 

This completes the proof. 
\end{proof}

\begin{lemma}\label{lem:chains} If $\mathcal{C}$ is a chain in  $(E_+,\leq)$ or in  $(E_-,\leq)$, then there exists $w$ that illuminates each element of $\mathcal{C}$.
\end{lemma}
\begin{proof}
Let $\mathcal{C}$ be a chain in  $(E_+,\leq)$ or in  $(E_-,\leq)$. 
We call a chain $c_1\leq c_2\leq \ldots\leq c_m$ in $(E_+,\leq)$ or in  $(E_-,\leq)$ maximal if it has length $n-1$. The chain $\mathcal{C}$ is contained in a maximal chain. As each coordinate permutation is an isometry of $(\mathbb{R}^{n-1},\|\cdot\|_H)$ and the map $x\mapsto -x$ is an isometry of $(\mathbb{R}^{n-1},\|\cdot\|_H)$, we may assume without loss of generality that 
$\mathcal{C}$ is contained in the maximal chain, 
\[
\mathcal{C}^*\colon (1,0,0,\ldots,0)\leq (1,1,0,\ldots, 0)\leq \ldots\leq (1,1,\ldots, 1,0)\leq (1,1,1,\ldots,1).
\]
Let $w\in\mathbb{R}^{n-1}$ be such that $w_1<w_2<\ldots<w_{n-1}<0$. 
Now if $x$ is the $k$-th element in the maximal chain and $k<n-1$, then  for all $\lambda>0$ sufficiently small  
\[
\|x+\lambda w\|_H = \left(\max_i x_i+\lambda w_i \right) \vee 0 - \left(\min_i x_i+\lambda w_i \right) \wedge 0 = 1+\lambda w_k - \lambda w_{k+1} <1.
\] 
On the other hand, if $x=(1,1,\ldots,1)$, then  clearly $\|x+\lambda w\|_H =1+\lambda w_{n-1}<1$ for all $\lambda>0$ small. Thus $w$ illuminates each element of $\mathcal{C}^*$ and we are done. 
\end{proof}

To proceed we need to recall a few classical results in the combinatorics of finite partially ordered sets, see \cite[Sections 9.1 and 9.2]{Ju}. Firstly, we recall Dilworth's Theorem, which says that if the maximum size of an antichain in a finite partially ordered set  $(P,\preceq)$ is $r$, then $P$ can be partitioned into $r$  disjoint chains. In the case where the partially ordered set is $(\{0,1\}^d,\leq)$, one can combine this result with Sperner's Theorem, which says that the maximum size of antichain in  $(\{0,1\}^d,\leq)$ is ${d\choose \lceil d/2\rceil}$. Thus, $(\{0,1\}^d,\leq)$ can be partitioned into ${d\choose \lceil d/2\rceil}$ disjoint chains. 

To obtain our result we need some more detailed information about the partitions. In particular, we need a result by De Bruijn,Tengbergen, Kruyswijk \cite{dBTK} concerning symmetric chains, see  also \cite[Theorem 9.3]{Ju}. 
A chain $x^1\leq \ldots\leq x^k$  in  $(\{0,1\}^{d},\leq)$ is said to be {\em symmetric} if 
\begin{enumerate}[(a)]
\item $(\sum_{j=1}^d x^m_j ) + 1=  \sum_{j=1}^d x^{m+1}_j$ for all $1\leq m<k$, i.e., $x^{m+1}$ is an immediate successor of  $x^m$,
\item $\sum_{j=1}^d x^k_j  = d- \sum_{j=1}^d x^{1}_j$.
\end{enumerate}
\begin{theorem}[De Bruijn,Tengbergen, Kruyswijk] \label{DBTK} The poset $(\{0,1\}^d,\leq)$ can be partitioned into ${d\choose \lceil d/2\rceil}$ disjoint symmetric chains. 
\end{theorem}

Let us now prove the main result of the paper. 
\begin{proof}[Proof of Theorem \ref{thm:main}]
First recall that  by Corollary \ref{lowerbnd} it suffices to show that  
$i(B_H)\leq {n\choose \lceil\frac{n}{2}\rceil}$, as $i(B_{\mathrm{v}})=i(B_H)$.   In other words, we only need to show that $\mathrm{ext}(B_H)$ can be illuminated by ${n\choose \lceil\frac{n}{2}\rceil}$ 
vectors.

There are two cases to consider: $n\geq 2$ even, and $n\geq 2$ odd. 

Let us first consider the case where $n\geq 2$ is even. 
By Dilworth's Theorem and Sperner's Theorem we know that the partially ordered set  $(\{0,1\}^{n-1},\leq)$ can be partitioned into ${n-1\choose \lceil \frac{n-1}{2}\rceil}$ disjoint chains. 
This implies that each of the partially ordered sets $(E_+,\leq)$ and $(E_-,\leq)$ can be partitioned into ${n-1\choose \lceil \frac{n-1}{2}\rceil}$ disjoint chains.  It now follows from Lemma \ref{lem:chains} that we need at most 
\[
{n-1\choose \lceil \frac{n-1}{2}\rceil}+{n-1\choose \lceil \frac{n-1}{2}\rceil}= {n-1\choose \lfloor \frac{n-1}{2}\rfloor}+{n-1\choose \lceil \frac{n-1}{2}\rceil} = {n\choose \frac{n}{2}}
\] 
distinct vectors to illuminate $\mathrm{ext}(B_H)$. This implies that $i(B_{\mathrm{v}})= i(B_H)\leq{n\choose \frac{n}{2}}$. 

Now suppose that $n\geq 2$ is odd. By Theorem \ref{DBTK} we know that  $(\{0,1\}^{n-1},\leq)$ can be partitioned into ${n-1\choose  \frac{n-1}{2}}$ disjoint symmetric chains. 

Let us consider such a symmetric chain decomposition, and let 
\[\mathcal{A}_k:=\{x\in \{0,1\}^{n-1}\colon \mbox{$\sum_i x_i =k$}\},\] which is an antichain of size ${n-1\choose k}$. Each element of $\mathcal{A}_{(n+1)/2}$ is contained in a distinct symmetric chain, and each of these chain contains an $x\in \{0,1\}^{n-1}$ with $\sum_i x_i =(n-1)/2$. Thus, the symmetric chain decomposition of $(\{0,1\}^{n-1},\leq)$ consists of 
\[
{n-1 \choose \frac{n+1}{2}}
\]
chains containing a vector $x$ with $\sum_i x_i =(n+1)/2$, and
\[
{n-1\choose \frac{n-1}{2}} - {n-1 \choose \frac{n+1}{2}}
\]
chains consisting of a single vector $x$ with $\sum_i x_i =(n-1)/2$. 

By deleting $(0,0,\ldots,0)$ from $\{0,1\}^{n-1}$ we obtain a partition of $(E_+,\leq )$ into disjoint chains. Let $\mathcal{S}$ be the set of vectors in $E_+$ which form a singleton chain and $\sum_i x_i = (n-1)/2$. So, 
\[
|\mathcal{S}| = {n-1\choose \frac{n-1}{2}} - {n-1 \choose \frac{n+1}{2}}.
\]

Now pair each $x\in E_+$ with $x'\in E_-$, where $x'_i=0$ if $x_i=1$, and $x'_i=-1$ if $x_i=0$. In this way we obtain a partition of $(E_-,\leq)$ into disjoint chains with $|\mathcal{S}|$ chains consisting of a single vector. In other words, for each $x\in \mathcal{S}$ we have that $x'\in E_-$ forms a singleton chain in the chain decomposition of $(E_-,\leq)$.

We know from Lemma \ref{lem:chains} that we can illuminate the ${n-1\choose \frac{n+1}{2}}$ chains in $(E_+,\leq)$  containing a vector $x$ with $\sum_i x_i = (n+1)/2$ using ${n-1\choose \frac{n+1}{2}}$ vectors. Likewise, we can  illuminate the corresponding  ${n-1\choose \frac{n+1}{2}}$ chains in $(E_-,\leq)$  with  
${n-1\choose \frac{n+1}{2}}$ vectors. So, it remains to illuminate the singleton chains in $(E_+,\leq)$ and $(E_-,\leq)$. 

Note that if we can illuminate each pair $\{x,x'\}$, with $x\in\mathcal{S}$ and $x'$ the corresponding vector in $E_-$, by a single vector, then we need at most 
\[
2{n-1\choose \frac{n+1}{2}}+ {n-1\choose \frac{n-1}{2}} - {n-1\choose \frac{n+1}{2}} = {n-1\choose \frac{n-1}{2}} +{n-1\choose \frac{n+1}{2}} = {n\choose \lceil \frac{n}{2}\rceil}
\]  
vectors to illuminate $\mathrm{ext}(B_H)$, and hence $i(B_{\mathrm{v}}) = i(B_H) \leq   {n\choose \lceil \frac{n}{2}\rceil}$ if $n\geq 2$ is odd. 

To see how this can be done we consider such a pair $\{x,x'\}$ with $x\in\mathcal{S}$ and let $I:=\{i\colon x_i =1\}$ and $J:=\{i\colon x_i=0\}$. So, $I=\{i\colon x'_i =0\}$ and $J =\{i\colon x'_i=-1\}$.  Now let $w\in\mathbb{R}^{n-1}$ be such that $w_i<0$ for all $i\in I$ and $w_i>0$ for all $i\in J$. 
Then for all $\lambda>0$ sufficiently small, 
\[
\|x+\lambda w\|_H = \max_{i\in I} (1+\lambda w_i) - 0 <1
\]
and 
\[
\|x'+\lambda w\|_H = 0 - \min_{i\in J} (-1 +\lambda w_i)<1.
\]
This shows that $w$ illuminates $x$ and $x'$, which completes the proof. 
\end{proof}

\end{document}